\RequirePackage[l2tabu, orthodox]{nag}

\documentclass[12pt]{amsart}
\usepackage{fullpage,url,amssymb,enumerate,colonequals}
\usepackage[all]{xy} % for commutative diagrams
\usepackage{comment}
\usepackage{graphicx}

\usepackage[OT2,T1]{fontenc}

\usepackage{color}
\def\Adm{\mathcal{A}dm}

\def\cA{\mathcal{A}}
\def\cB{\mathcal{B}}
\def\bC{\mathbb{C}}
\def\cC{\mathcal{C}}

\def\cE{\mathcal{E}}

\def\cH{\mathcal{H}}

\def\bL{\mathbb{L}}
\def\cL{\mathcal{L}}
\def\cM{\mathcal{M}}

\def\cO{\mathcal{O}}

\def\cP{\mathcal{P}}

\def\cT{\mathcal{T}}

\def\cX{\mathcal{X}}

\def\barA{\overline{\cA}}
\def\barC{\overline{\cC}}
\def\barE{\overline{\cE}}

\def\barM{\overline{\cM}}

\def\wt{\widetilde}

\DeclareMathOperator{\alg}{alg}

\DeclareMathOperator{\ct}{ct}

\DeclareMathOperator{\Def}{Def}
\DeclareMathOperator{\ev}{ev}
\DeclareMathOperator{\Frac}{Frac}

\DeclareMathOperator{\id}{id}

\DeclareMathOperator{\Pic}{Pic}
\DeclareMathOperator{\pr}{pr}

\DeclareMathOperator{\red}{red}

\DeclareMathOperator{\Spec}{Spec}

\DeclareMathOperator{\trunc}{trunc}

\DeclareMathOperator{\vir}{vir}

\newcommand{\Hom}{\mathrm{Hom}}

\newcommand{\NL}{{\mathsf{NL}}}
\newcommand{\wNL}{{\widetilde{\mathsf{NL}}}}
\newcommand{\Tor}{{\mathsf{Tor}}}

\newtheorem{thm}{Theorem}[section]
\newtheorem{lem}[thm]{Lemma}
\newtheorem{cor}[thm]{Corollary}
\newtheorem{prop}[thm]{Proposition}

\theoremstyle{definition}
\newtheorem{rem}[thm]{Remark}

\newtheorem{defn}[thm]{Definition}

\makeatletter
\g@addto@macro\bfseries{\boldmath} % This makes math in section titles bold.
\makeatother

\usepackage{microtype}

\begin{document}

\title{$d$-elliptic loci and the Torelli map}

\author{Fran\c cois Greer}
\address{Michigan State University, Department of Mathematics, 619 Red Cedar Rd
\hfill \newline\texttt{}
 \indent East Lansing, MI 48824}
 \email{{\tt greerfra@msu.edu}}

\author{Carl Lian}

\address{Tufts University, Department of Mathematics, 177 College Ave
\hfill \newline\texttt{}
 \indent Medford, MA 02155} 
 \email{{\tt Carl.Lian@tufts.edu}}
 
\date{\today}

\begin{abstract}
We show that two natural cycle classes on the moduli space of compact type stable maps to a varying elliptic curve agree. The first is the virtual fundamental class from Gromov-Witten theory, and the second is the Torelli pullback of the special cycle on $\cA_g$ of principally polarized abelian varieties admitting an elliptic isogeny factor.
\end{abstract}

\maketitle

%****************************************************************************

\section{Introduction}

Let $\cM_g$ be the moduli space of curves. This note is concerned with the loci of curves $C$ admitting a cover $f:C\to E$, where $E$ is an elliptic curve and $f$ has degree $d$. The Hurwitz space $\cH_{g/1,d}$ of such covers can be compactified on the one hand by the space $\Adm_{g/1,d}$ of admissible covers \cite{hm,acv}, or alternatively by a relative moduli space of log stable maps $\barM_{g,1}(\overline{\cE},d)$ to the fibers of the universal stable elliptic curve.

The resulting classes $[\Adm_{g/1,d}]\in A_{2g-2}(\barM_g)$ are particularly interesting: they give cycle-valued quasi-modular forms in low genus \cite{lian_d-ell}, but on the other hand are non-tautological in general \cite{gp,vz,lian_nontaut}, so it is difficult to write down formulas for their classes. For covers of a \emph{fixed} elliptic curve, quasi-modularity phenomena have been established in all genera \cite{dijkgraaf,eo,okpan,obpix}.

A basic observation is that if $f:C\to E$ is a $d$-elliptic cover of smooth curves, then the induced map $h:E\to J(C)$ on Jacobians has the property that the pullback of the principal polarization on $J(C)$ has degree $d$ on $E$. In this case, $J(C)$ is isogenous to the product of $E$ and a polarized abelian variety $Y$ of dimension $g-1$. Therefore, it is natural to consider the ``Noether-Lefschetz locus'' $\wNL_{g,d}\to\cA_g$ of principally polarized abelian varieties along with a map of degree $d$ from an elliptic curve, and its pullback under the Torelli map $\Tor:\cM_{g}^{\ct}\to \cA_g$.

It is somewhat more convenient to pull back the Noether-Lefschetz loci by the \emph{pointed} Torelli map $\Tor_1:\cM_{g,1}^{\ct}\to \cA_g$, although we discuss the unpointed case in Remark \ref{rem_unpointed}. 

\begin{thm}\label{thm_pullback_to_mg1}
We have a Cartesian diagram
\begin{equation*}
\xymatrix{
\cM^{\ct,q}_{g,1}(\cE,d) \ar[r]^(0.6){\wt{\Tor}_1} \ar[d]_{\epsilon} & \wNL_{g,d} \ar[d]^{\epsilon} \\
\cM_{g,1}^{\ct}  \ar[r]^(0.54){\Tor_1} & \cA_g
}
\end{equation*}
where $\cM^{\ct,q}_{g,1}(\cE,d)$ is the moduli space of pointed stable maps $f:(C,p)\to (E,q)$ to a (varying) smooth elliptic curve $(E,q)$, for which $(C,p)$ is of compact type and $f(p)=q$.
\end{thm}

The statement of Theorem \ref{thm_pullback_to_mg1} has also appeared previously in the notes \cite{pand_notes}.

Unlike the space of admissible covers, the space $\cM^{\ct,q}_{g,1}(\cE,d)$ fails to have the expected dimension of $2g-1$ in general. For example, if $C=C_{g-1}\cup E'$ is the stable curve obtained by gluing a curve $C_{g-1}\ni p$ of genus $g-1$ and an elliptic tail $E'$ at a node (away from $p$), then there exist stable maps $f:(C,p)\to (E,q)$ where $f$ contracts the entirety of $C_{g-1}$ to $q$, and restricts to an isogeny of degree $d$ on $E'$. However, $C$ moves in a family of dimension $3g-3$.

On the other hand, the space $\cM^{\ct,q}_{g,1}(\cE,d)$ carries a virtual fundamental class obtained by Gysin pullback in the Cartesian diagram
\begin{equation*}
    \xymatrix{
    \cM_{g,1}^{\ct,q}(\cE,d) \ar[r]^{i} \ar[d] & \cM_{g,1}^{\ct}(\cE,d) \ar[d]^{\ev} \\
    \cM_{1,1} \ar[r]^{q} & \cE 
    }
\end{equation*}
where $q:\cM_{1,1}\to\cE$ denotes the tautological section and $\cM^{\ct}_{g,1}(\cE,d)$ denotes the space of pointed stable maps which are no longer required to satisfy the property $f(p)=q$. The virtual class $[\cM^{\ct}_{g,1}(\cE,d)]^{\vir}$ agrees with the usual virtual class in Gromov-Witten theory upon restriction to each fixed elliptic target. We show that the Gysin pullback of the fundamental class of $\wNL_{g,d}$ by the Torelli map agrees precisely with the virtual class on $\cM^{\ct,q}_{g,1}(\cE,d)$.

\begin{thm}\label{main_thm}
We have $\Tor_1^{!}([\wNL_{g,d}])=q^{!}([\cM^{\ct}_{g,1}(\cE,d)]^{\vir})$ in $A_{2g-1}(\cM^{\ct}_{g,1}(\cE,d))$.
\end{thm}

Theorems \ref{thm_pullback_to_mg1} and \ref{main_thm} provide a bridge between the intersection theory of $\cA_g$ and the Gromov-Witten theory of maps to a varying elliptic curve. A precise
connection of generating series is proposed in \cite{pand_notes}. Projections, as defined in \cite{cmop}, of the Noether-Lefschetz cycle classes have been computed by Iribar L\'{o}pez \cite{il}. The forthcoming work \cite{il} also includes proofs of some of our auxiliary results, e.g. Lemma \ref{maps_proper} and Proposition \ref{NL_normal_bundles}. The Gromov-Witten theory of maps to a varying elliptic curve are also related to the genus 1 Gromov-Witten theory of the Hilbert scheme of points of $\mathbb{C}^2$ by \cite{mnop2,pt}.

It is natural to ask the extent to which Theorems 1.1 and 1.2 can be extended to the boundary. That is, for some compactification $\barA_g$ of $\cA_g$ admitting an extension of the Torelli map, whether the class of the closure (properly interpreted) of $\wNL_{g,d}$ pulls back to a virtual class on $\barM^{q}_{g,1}(\barE,d)$ or a closely related space. Such a statement would facilitate further calculations in Gromov-Witten theory. Formulating an analog of Theorem 1.1 at the boundary of $\cA_g$ already seems to be subtle. This question is currently under investigation.

\subsection{Acknowledgements} We thank Rahul Pandharipande for helpful discussions and for encouraging us to write this note, and thank Aitor Iribar L\'{o}pez for comments on a draft. The second named author is grateful to Johan de Jong for originally suggesting the link between $d$-elliptic loci on $\cM_g$ and $\cA_g$ during his thesis work. Finally, we thank the referee for their careful reading and helpful comments. F.G. was supported by NSF grant DMS-2302548. C.L. was supported by an NSF postdoctoral fellowship, grant DMS-2001976, and an AMS-Simons travel grant.

\section{Moduli of stable maps and virtual classes}

In this section, we construct virtual fundamental classes on spaces of stable maps to a varying elliptic curve, globalizing the usual constructions on $\barM_{g,n}(E,d)$ of stable maps to a fixed elliptic curve. 

\subsection{Stable maps to a varying elliptic curve}\label{kontsevich_family}

\begin{defn}
Let $\pi:\cE\to\cM_{1,1}$ be the universal family of elliptic curves with section $q:\cM_{1,1}\to \cE$.
\end{defn}

\begin{defn}
Let $\barM_{g,n}(\cE,d)$ denote the space of $n$-pointed degree $d$ stable maps $C\to E$, where $E$ is a smooth elliptic curve. Formally, a map $S\to \barM_{g,n}(\cE,d)$ corresponds to a family of pre-stable curves $\cC_S\to S$, a family of smooth elliptic curves $\cE_S\to S$, and a map $f:\cC_S\to\cE_S$ that is a pointed stable map on every fiber.

Let $\mu:\barM_{g,n}(\cE,d)\to\cM_{1,1}$ be the map remembering the target elliptic curve $E$. Let $\epsilon:\barM_{g,n}(\cE,d)\to\barM_{g,n}$ be the map remembering the source curve, with unstable components contracted.

Let $\pi_{g,n}:\barC_{g,n}(\cE,d)\to \barM_{g,n}(\cE,d)$ be the universal curve, and let $f:\barC_{g,n}(\cE,d)\to\cE$ be the universal stable map.

For $i=1,2,\ldots,n$, let $\ev_i:\barM_{g,n}(\cE,d)\to\cE$ denote the evaluation map remembering the image of the $i$-th marked point of a stable map on $\cE$.
\end{defn}

The fibers of $\mu$ are the usual spaces of stable maps $\barM_{g,n}(E,d)$, and the maps $\epsilon,\pi_{g,n},f$ similarly globalize the usual constructions.

\begin{lem}\label{domain_automatically_stable}
    Let $f:(C,p_1,\ldots,p_n)\to E$ be a point of $\barM_{g,n}(\cE,d)$. Then, the pointed curve $(C,p_1,\ldots,p_n)$ is stable.
\end{lem}

\begin{proof}
    The pointed curve $(C,p_1,\ldots,p_n)$ can only fail to be stable if it contains a rational component mapping non-trivially to $E$. However, this is impossible when $E$ is a smooth elliptic curve.
\end{proof}

Analogously to the absolute setting \cite{behrend}, the relative virtual fundamental class $[\barM_{g,n}(\cE,d)]^{\vir}\in A_{2g-1+n}(\barM_{g,n}(\cE,d))$ is constructed from a relative perfect obstruction theory  $$(R(\pi_{g,n})_{*}f^{*}T_{\cE/\cM_{1,1}})^{\vee}\to \bL_{\barM_{g,n}(\cE,d)/(\barM_{g,n}\times\cM_{1,1})}$$
whose construction we recall. First, we have a map 
\begin{align*}
    f^{*}\bL_{\cE/\cM_{1,1}}&\to\bL_{\barC_{g,n}(\cE,d)/\cM_{1,1}}\\
    &\to \bL_{\barC_{g,n}(\cE,d)/(\barC_{g,n}\times\cM_{1,1})}\\
    &\cong (\pi_{g,n})^{*}\bL_{\barM_{g,n}(\cE,d)/(\barM_{g,n}\times\cM_{1,1})},
\end{align*} 
where $\barC_{g,n}$ is the universal curve over $\barM_{g,n}$. Note that we have used that the diagram
\begin{equation*}
    \xymatrix{
    \barC_{g,n}(\cE,d) \ar[r] \ar[d]^{\pi_{g,n}} & \barC_{g,n} \times\cM_{1,1}\ar[d] \\
    \barM_{g,n}(\cE,d) \ar[r]^(0.45){(\epsilon,\mu)} & \barM_{g,n}\times\cM_{1,1}
    }
\end{equation*}
is Cartesian, which is true by Lemma \ref{domain_automatically_stable}. When considering more general targets, one needs to replace $\barM_{g,n}$ with the stack of pre-stable curves $\mathfrak{M}_{g,n}$.

We therefore obtain
\begin{align*}
    (R(\pi_{g,n})_{*}f^{*}T_{\cE/\cM_{1,1}})^{\vee}&=(R(\pi_{g,n})_{*}f^{*}\bL_{\cE/\cM_{1,1}}^{\vee})^{\vee}\\
    &\to (R(\pi_{g,n})_{*}((\pi_{g,n})^{*}\bL_{\barM_{g,n}(\cE,d)/(\barM_{g,n}\times\cM_{1,1})})^{\vee})^{\vee}\\
    &\cong \bL_{\barM_{g,n}(\cE,d)/(\barM_{g,n}\times\cM_{1,1})}\otimes^{L} (R(\pi_{g,n})_{*}\cO_{\barC_{g,n}(\cE,d)})^{\vee}\\
    &\to \bL_{\barM_{g,n}(\cE,d)/(\barM_{g,n}\times\cM_{1,1})}
\end{align*}

By \cite[Proposition 7.2]{bf}, pulling back the resulting virtual class $[\barM_{g,n}(\cE,d)]^{\vir}$ to the fibers $\barM_{g,n}(E,d)$ of $\mu$ recovers the usual virtual fundamental classes in Gromov-Witten theory.

\subsection{1-pointed stable maps to an elliptic curve}\label{1-pointed_sec}

We now specialize the discussion to the case $n=1$. We will denote the evaluation map $\ev_1:\barM_{g,1}(\cE,d)\to \cE$ simply by $\ev$.

\begin{defn}\label{i_def}
Define the closed immersion $i:\barM_{g,1}^q(\cE,d)\to\barM_{g,1}(\cE,d)$ via the Cartesian diagram
\begin{equation*}
    \xymatrix{
    \barM_{g,1}^q(\cE,d) \ar[r]^{i} \ar[d] & \barM_{g,1}(\cE,d) \ar[d]^{\ev} \\
    \cM_{1,1} \ar[r]^{q} & \cE 
    }
\end{equation*}
The substack $\barM_{g,1}^q(\cE,d)$ parametrizes 1-pointed stable maps whose marked point maps to the origin of its target elliptic curve. 

Let $\pi_g^q:\barC_{g,1}^q(\cE,d)\to \barM_{g,1}^q(\cE,d)$ be the universal curve and $f^q:\barC_{g,1}^q(\cE,d)\to\cE$ be the universal map.

Let $\epsilon\circ i=\epsilon^q:\barM_{g,1}^q(\cE,d)\to\barM_{g,1}$ and $\mu\circ i=\mu^q:\barM_{g,1}^q(\cE,d)\to\cM_{1,1}$ be the maps remembering the source and target curves, respectively.

Let 
\begin{equation*}
    \pr:\barM_{g,1}(\cE,d)\to \barM_{g,1}^q(\cE,d)
\end{equation*}
be the map sending a pointed stable map $f:(C,p)\to (E,q)$ to the map $f-f(p)$, where ``$-f(p)$'' denotes post-composition with the translation on $E$ sending $f(p)$ to the origin $q$. 
\end{defn}
We then have a fibered product decomposition 
\begin{equation*}
    (\pr,\ev):\barM_{g,1}(\cE,d)\cong \barM_{g,1}^q(\cE,d)\times_{\cM_{1,1}}\cE.
\end{equation*}

We now construct a virtual fundamental class on $\barM_{g,1}^q(\cE,d)$; we first explain the idea. The deformations and obstructions of a stable map $[f:C\to E]\in \barM_{g,1}(\cE,d)$ with fixed source and target are given by $H^i(C,f^{*}T_E)$ with $i=0,1$, respectively. The 1-dimensional space of deformations $H^0(C,f^{*}T_E)$ comes from post-composing with translations on $E$. On the substack $\barM_{g,1}^q(\cE,d)$, these deformations go away; indeed, the condition that the marked point of $C$ map to the origin of $E$ is cut out locally by the vanishing of this deformation direction. Accordingly, the obstruction theory on $\barM_{g,1}^q(\cE,d)$ is obtained by truncating that of $\barM_{g,1}(\cE,d)$ to remove the one-dimensional deformations of $f$.

Formally, as on $\barM_{g,1}(\cE,d)$, we have a map of complexes on $\barM^q_{g,1}(\cE,d)$:
\begin{equation*}
(R(\pi_{g}^q)_{*}(f^q)^{*}T_{\cE/\cM_{1,1}})^{\vee}\to \bL_{\barM_{g,1}^q(\cE,d)/\barM_{g,1}\times\cM_{1,1}}.
\end{equation*}
We may compose on the left with the canonical truncation $\tau_{\ge -1}$ to obtain a map  
\begin{equation*}
    e_{\trunc}:(R^1(\pi_{g}^q)_{*}(f^q)^{*}T_{\cE/\cM_{1,1}})^{\vee}[1]\to \bL_{\barM_{g,1}^q(\cE,d)/\barM_{g,1}\times\cM_{1,1}}
\end{equation*}

\begin{prop}\label{trunc_perfobs}
    The map $e_{\trunc}$ is a perfect obstruction theory, and thus gives rise to a \textbf{truncated} virtual fundamental class $[\barM_{g,1}^q(\cE,d)]^{\trunc}\in  A_{2g-1}(\barM_{g,1}^q(\cE,d))$.
\end{prop}

\begin{proof}
We need to show that $e_{\trunc}$ is an isomorphism in degree 0 and surjection in degree $-1$. For any $f\in \barM_{g,1}^q(\cE,d)$, we have $H^0(C,f^{*}T_E(-p))=0$, so it follows that the map 
\begin{equation*}
(\epsilon^q,\mu^q):\barM_{g,1}^q(\cE,d)\to\barM_{g,1}\times\cM_{1,1}
\end{equation*}
is formally unramified. Thus, the restriction of $e_{\trunc}$ to degree 0 is the isomorphism of zero objects, and the restriction to degree $-1$ gives, as in the case of the Gromov-Witten obstruction theory of $\barM_{g,1}(E,d)$, the surjection of sheaves
\begin{equation*}
    (R^1(\pi_{g}^q)_{*}(f^q)^{*}T_{\cE/\cM_{1,1}})^{\vee}\to \bL^{\ge-1}_{\barM_{g,1}^q(\cE,d)/\barM_{g,1}\times\cM_{1,1}}\cong C_{\barM_{g,1}^q(\cE,d)/\barM_{g,1}\times\cM_{1,1}},
\end{equation*}
where $C_{\barM_{g,1}^q(\cE,d)/\barM_{g,1}\times\cM_{1,1}}$ is the conormal sheaf.
\end{proof}

In fact, the truncated class we have constructed is precisely the pullback of the virtual class on the ambient space $\barM_{g,1}(\cE,d)$.

\begin{prop}\label{trunc_gw_comparison}
We have
$$[\barM_{g,1}^q(\cE,d)]^{\trunc}=q^{!}([\barM_{g,1}(\cE,d)]^{\vir}),$$
where $q^{!}$ denotes Gysin pullback in the diagram of Definition \ref{i_def}.
\end{prop}

\begin{proof}
We apply \cite[Proposition 7.5]{bf} to the right half of the Cartesian diagram
\begin{equation*}
    \xymatrix{
    \barM^q_{g,1}(\cE,d) \ar[d]_{(\epsilon^q,\mu^q)} \ar[r]^{i} \ar[d] & \barM_{g,1}(\cE,d) \ar[r]^{\pr} \ar[d]_{(\epsilon,\ev)} & \barM^q_{g,1}(\cE,d) \ar[d]^{(\epsilon^q,\mu^q)} \\
    \barM_{g,1}\times\cM_{1,1} \ar[r]^{(\id,q)} & \barM_{g,1}\times\cE \ar[r]^(0.43){(\id,\pi)} & \barM_{g,1}\times\cM_{1,1}.
    }
\end{equation*}
of stacks over $\barM_{g,1}\times\cM_{1,1}$. It suffices to construct a map of distinguished triangles
\begin{equation*}
    \xymatrix{
    \pr^{*}(R^1(\pi_{g}^q)_{*}(f^q)^{*}T_{\cE/\cM_{1,1}})^{\vee}[1] \ar[r] \ar[d] & (R(\pi_{g,1})_{*}f^{*}T_{\cE/\cM_{1,1}})^{\vee} \ar[r] \ar[d] & (\epsilon,\ev)^{*}\bL_{(\barM_{g,1}\times\cE)/(\barM_{g,1}\times\cM_{1,1})} \ar[d] \\
    \pr^{*}\bL_{\barM^q_{g,1}(\cE,d)/(\barM_{g,1}\times
    \cM_{1,1})} \ar[r] & \bL_{\barM_{g,1}(\cE,d)/(\barM_{g,1}\times\cM_{1,1})} \ar[r] &  \bL_{\barM_{g,1}(\cE,d)/\barM^q_{g,1}(\cE,d)}
    }
\end{equation*}
This will show that $(\id,\pi)^{!}([\barM_{g,1}^q(\cE,d)]^{\trunc})=[\barM_{g,1}(\cE,d)]$, and applying $q^!$ to both sides yields $[\barM_{g,1}^q(\cE,d)]^{\trunc}=q^{!}([\barM_{g,1}(\cE,d)]^{\vir})$.
%\FG{Yes, sounds good.}

The square on the left arises from the Cartesian diagram
\begin{equation*}
    \xymatrix{
    \barC_{g,1}(\cE,d) \ar[r]^{\wt{\pr}} \ar[d]_{\pi_{g,1}} & \barC^q_{g,1}(\cE,d) \ar[d]^{\pi_{g}^q} \ar[r]^(0.63){f^q} & \cE  \\
    \barM_{g,1}(\cE,d) \ar[r]^{\pr} & \barM^q_{g,1}(\cE,d) 
     & 
     }
\end{equation*}
where the horizontal arrows are flat and whose pullbacks therefore commute with (derived) pushforward; we obtain the compatibility
\begin{equation*}
    \xymatrix{
    \pr^{*}(R(\pi_{g}^q)_{*}(f^q)^{*}T_{\cE/\cM_{1,1}})^{\vee} \ar[r]^(0.53){\sim} \ar[d] & (R(\pi_{g,1})_{*}f^{*}T_{\cE/\cM_{1,1}})^{\vee} \ar[d] \\
    \pr^{*}\bL_{\barM^q_{g,1}(\cE,d)/(\barM_{g,1}\times\cM_{1,1})} \ar[r] & \bL_{\barM_{g,1}(\cE,d)/(\barM_{g,1}\times\cM_{1,1})}}
\end{equation*}
and then pre-compose in the top left with truncation.

%\CL{rewrote this last part slightly}
Now, the mapping cone of $\pr^{*}(R^1(\pi_{g}^q)_{*}(f^q)^{*}T_{\cE/\cM_{1,1}})^{\vee}[1] \to (R(\pi_{g,1})_{*}f^{*}T_{\cE/\cM_{1,1}})^{\vee}$ is given by the canonical truncation $$\tau_{\ge0}(R(\pi_{g,1})_{*}f^{*}T_{\cE/\cM_{1,1}})^{\vee}=((\pi_{g,1})_{*}f^{*}T_{\cE/\cM_{1,1}})^{\vee},$$
where on the right hand side we have a line bundle in degree 0. Therefore, we have a map of distinguished triangles
\begin{equation*}
    \xymatrix{
    \pr^{*}(R^1(\pi_{g}^q)_{*}(f^q)^{*}T_{\cE/\cM_{1,1}})^{\vee}[1] \ar[r] \ar[d] & (R(\pi_{g,1})_{*}f^{*}T_{\cE/\cM_{1,1}})^{\vee} \ar[r] \ar[d] & ((\pi_{g,1})_{*}f^{*}T_{\cE/\cM_{1,1}})^{\vee} \ar[d] \\
    \pr^{*}\bL_{\barM^q_{g,1}(\cE,d)/(\barM_{g,1}\times
    \cM_{1,1})} \ar[r] & \bL_{\barM_{g,1}(\cE,d)/(\barM_{g,1}\times\cM_{1,1})} \ar[r] &  \bL_{\barM_{g,1}(\cE,d)/\barM^q_{g,1}(\cE,d)}
    }
\end{equation*}

It now remains to show that we may replace the object $((\pi_{g,1})_{*}f^{*}T_{\cE/\cM_{1,1}})^{\vee}$ in the above diagram with $(\epsilon,\ev)^{*}\bL_{(\barM_{g,1}\times\cE)/(\barM_{g,1}\times\cM_{1,1})}$. It suffices to construct a commutative diagram
\begin{equation*}
    \xymatrix{
    ((\pi_{g,1})_{*}f^{*}T_{\cE/\cM_{1,1}})^{\vee} \ar[rr] \ar[rd]  &  & (\epsilon,\ev)^{*}\bL_{(\barM_{g,1}\times\cE)/(\barM_{g,1}\times\cM_{1,1})} \ar[ld]  \\
     & \bL_{\barM_{g,1}(\cE,d)/\barM^q_{g,1}(\cE,d)} &
    }
\end{equation*}
where the horizontal arrow is an isomorphism and the vertical maps are the given ones. 

Note that the canonical map $(\epsilon,\ev)^{*}\bL_{(\barM_{g,1}\times\cE)/(\barM_{g,1}\times\cM_{1,1})} \to \bL_{\barM_{g,1}(\cE,d)/\barM^q_{g,1}(\cE,d)}$ is an isomorphism of line bundles in degree zero. Moreover, the composition
\begin{equation*}
(R(\pi_{g,1})_{*}f^{*}T_{\cE/\cM_{1,1}})^{\vee}\to \bL_{\barM_{g,1}(\cE,d)/(\barM_{g,1}\times\cM_{1,1})}\to \bL_{\barM_{g,1}(\cE,d)/\barM^q_{g,1}(\cE,d)}
\end{equation*}
induces isomorphisms on cohomology sheaves in degree 0. Indeed, the first map is a perfect obstruction theory, and the complex $\pr^{*}\bL_{\barM^q_{g,1}(\cE,d)/(\barM_{g,1}\times\cM_{1,1})}$ is zero in degree $0$ and above. Thus, the map $ ((\pi_{g,1})_{*}f^{*}T_{\cE/\cM_{1,1}})^{\vee}\to  \bL_{\barM_{g,1}(\cE,d)/\barM^q_{g,1}(\cE,d)}$ of mapping cones is also an isomorphism of line bundles in degree zero. The horizontal isomorphism $((\pi_{g,1})_{*}f^{*}T_{\cE/\cM_{1,1}})^{\vee} \to (\epsilon,\ev)^{*}\bL_{(\barM_{g,1}\times\cE)/(\barM_{g,1}\times\cM_{1,1})}$ above may then be obtained by composition. This completes the proof.

% On the other hand, the canonical map 
%  \begin{equation*}
%      (\epsilon,\ev)^{*}\bL_{(\barM_{g,1}\times\cE)/(\barM_{g,1}\times\cM_{1,1})}\to \bL_{\barM_{g,1}(\cE,d)/\barM^q_{g,1}(\cE,d)}
%  \end{equation*}
%  is also an isomorphism of line bundles in degree 0.
\end{proof}

\section{$d$-elliptic maps of abelian varieties}

\begin{defn}
Let $X$ be a principally polarized abelian variety (ppav) and $E$ an elliptic curve. We say that a map $h:X\to E$ (resp. a map $h:E\to X$) of abelian varieties has degree $d$ if the composite map $h\circ h^{\vee}:E\to E$ (resp. $h^{\vee}\circ h$) is given by multiplication by $d$. We say that the map $h:X\to E$ is \emph{$d$-elliptic}.
\end{defn}

Note that $\deg(h)=\deg(h^{\vee})$. The degree of $h:E\to X$ is also equal to the degree of the pullback of the principal polarization on $X$ to $E$. We will pass freely between the map $h:X\to E$ and its dual $h^\vee:E\to X$, identifying the two ppavs with their duals.

\begin{defn}
Let $\wNL_{g,d}$ be the stack of maps $h:X\to E$ of degree $d$, where $X$ is a ppav of dimension $g$ and $E$ is an elliptic curve. 

Abusing notation, let $\epsilon:\wNL_{g,d}\to\cA_g$ and $\mu:\wNL_{g,d}\to\cA_1$ be the maps remembering the source and target, respectively.
\end{defn}

\begin{rem}
The letters ``NL'' stand for ``Noether-Lefschetz,'' as the (dual) map $h^\vee: E\to X$ gives an exceptional algebraic cycle on $X$. We use the notation $\wNL_{g,d}$ to distinguish from the spaces $\NL_{g,d}$ (see \cite{pand_notes}) of \emph{immersions} $h:E\to X$ of degree $d$.
\end{rem}

The moduli space $\wNL_{g,d}$ may be constructed as a closed substack of the relative Hom stack $\Hom(\cX_g,\cX_1)\to\cA_g\times\cA_1$, where $\pi_i:\cX_i\to\cA_i$ are the universal families.
Let $\cX_{\NL,g}=\cX_g\times_{\cA_g} \wNL_{g,d}$ be the pullback of the universal families over $\wNL_{g,d}$, and let $h_{\NL}:\cX_{\NL,g} \to \cX_{1}$ be the universal $d$-elliptic map.

\begin{lem}\label{maps_proper}
    The morphisms $\epsilon:\wNL_{g,d}\to\cA_g$ and $(\epsilon,\mu):\wNL_{g,d}\to\cA_g\times \cA_1$ are proper.
\end{lem}

\begin{proof}
    Let $t:\Spec(R)\to \cA_g$ be a morphism corresponding to a family $\wt{X}/\Spec(R)$ of ppavs over a discrete valuation ring $R$. Let $\Frac(R)=K$ be the fraction field of $R$, and let $\wt{t}:\Spec(K)\to \wNL_{g,d}$ be a point lifting $t$ corresponding to a $d$-elliptic map $h:E_K\to \cX_K$. The map $h$ factors as $E_K\to E'_K\hookrightarrow \cX_K$, where $E_K\to E'_K$ is an isogeny of elliptic curves and $E'_K\hookrightarrow \cX_K$ is a closed immersion.

    Let $\cE'$ be the closure of $E'_K$ in $\cX$, which is smooth over $\Spec(R)$, because the special fiber $X_0$ contains no rational curves. Then, possibly after a finite base change, the isogeny $E_K\to E'_K$ may be extended to an isogeny $\cE\to \cE'$ over $\Spec(R)$, and the composition $\cE\to \cX$ extends $h$.
    
    The properness of $(\epsilon,\mu)$ is similar. In this case, the smooth elliptic curve $\cE/\Spec(R)$ obtained as above will automatically agree with that coming from a given morphism $t:\Spec(R)\to \cA_g\times\cA_1$.
\end{proof}

\subsection{Deformation Theory}

We next study the map $(\epsilon,\mu):\wNL_{g,d}\to \cA_g\times\cA_1$ locally. Throughout this section, the word ``deformation'' refers to a first-order deformation.

Let $X$ be a polarized abelian variety (not necessarily principally polarized) and let $\theta_X\in H^1(X,\Omega_X)$ be the class associated to the polarization. Recall that the space of deformations of $X$ as a variety is given by $H^1(X,T_X)$, and the space $\Def(X)$ of deformations of $X$ as a polarized abelian variety is given by the kernel $\Def(X):=H^1(X,T_X)^{\alg}$ of the surjection 
\begin{equation*}
H^1(X,T_X)\to H^2(X,\cO_X)
\end{equation*}
induced by cup product with $\theta_X$, see \cite[Theorem 3.3.11]{sernesi}.

Let $h:X\to E$ be a $d$-elliptic map, where $X$ is a ppav, and let $\Def(h)$ be its space of deformations. The goal of this section is to prove:
\begin{prop}\label{Def(h)}
    There exists a canonical short exact sequence
    \begin{equation*}
        0\to \Def(h) \to \Def(X)\oplus\Def(E) \to H^1(X,h^{*}T_E)\to 0.
    \end{equation*}
\end{prop}

This will globalize to the following statement:

\begin{prop}\label{NL_normal_bundles}
The map $(\epsilon,\mu):\wNL_{g,d}\to \cA_g\times \cA_1$ is unramified, with normal bundle 
% \begin{equation*}
%     R^1(\pi_{\cX})_{*}h_{\NL}^{*}T_{\cX_{1}/\wNL_{g,d}}=R^1(\pi_{\cX})_{*}T_{\cX_1/\cA_1}
% \end{equation*}
 \begin{equation*}
     R^1(\pi_{\NL,g})_{*}h_{\NL}^{*}T_{\cX_1/\cA_1},
 \end{equation*}
where $\pi_{\NL,g}:\cX_{\NL,g}\to \wNL_{g,d}$ is the pullback of the universal family $\pi_g:\cX_g\to \cA_g$ by $\epsilon:\wNL_{g,d}\to\cA_g$.
\end{prop}

We first define the maps appearing in Proposition \ref{Def(h)}. The map $\Def(h) \to \Def(X)\oplus\Def(E)$ is the canonical one remembering from the data of a deformation of $h$ the data of deformations of the source and target. The map $\Def(X)\to H^1(X,h^{*}T_E)$ is given by restriction of the map $H^1(X,T_X)\to H^1(X,h^{*}T_E)$ induced by $T_X\to h^{*}T_E$ to $\Def(X)=H^1(X,T_X)^{\alg}$. The map $\Def(E)=H^1(E,T_E)\to H^1(X,h^{*}T_E)$ is \emph{$-d$ times} the canonical one 
\begin{equation*}
    H^1(E,T_E)\to H^1(E,h_{*}h^{*}T_E)\to H^1(X,h^{*}T_E).
\end{equation*}

Let $\iota:Y\hookrightarrow X$ be the connected component of $\ker(h)$ containing $0\in X$, which inherits a (typically non-principal) polarization $\theta_Y\in H^1(Y,T_Y^{\vee})$ upon pullback. The map $h$ also induces an isogeny $j:Y\times E\to X$, where $j=(\iota,h^\vee)$; here we identify $E,X$ with their duals. The pullback of the principal polarization on $X$ is the \emph{degree $d$} polarization $d\cdot\theta_E$ on $E$ and the polarization $\theta_Y$ on $Y$.

To establish the exactness in Proposition \ref{Def(h)}, we will construct a commutative diagram
\begin{equation}\label{defdiagram}
    \xymatrix{ 0 \ar[r] & \Def(Y)\oplus\Def(E) \ar[r] & \Def(Y\times E) \oplus \Def(E) \ar[r] & H^1(Y\times E,T_E) \ar[r] & 0\\
     & \Def(h) \ar[u] \ar[r] & \Def(X)\oplus \Def(E) \ar[r] \ar[u] & H^1(X,h^{*}T_E) \ar[u] & 
     }
\end{equation}
where the top row is exact, the vertical arrows are isomorphisms, and the bottom row consists of the maps defined above. Here, $\Def(Y\times E)$ is the deformation space of $Y\times E$ as a polarized abelian variety, with polarization given by $j^{*}(\theta_X)=(\theta_Y,d\cdot\theta_E)$ as described above.

We next define the maps appearing in the top row of \eqref{defdiagram}. The map $\Def(Y)\oplus\Def(E)\to\Def(E)$ is simply the projection. The map $\Def(Y)\oplus\Def(E)\to\Def(Y\times E)$ deforms $Y\times E$ by deforming $Y$ and $E$ separately according to the given deformation of $Y$ and \emph{$d$ times} the given deformation of $E$. The map $\Def(Y\times E)\to H^1(Y\times E,T_E)$ is induced by the projection $T_{Y\times E}\to T_E$, and the map $\Def(E)=H^1(E,T_E)\to H^1(Y\times E,T_E)$ is \emph{$-d$ times} the one induced by the projection $Y\times E\to E$.

\begin{lem}
The sequence
\begin{equation*}
    0 \to \Def(Y)\oplus\Def(E) \to \Def(Y\times E) \oplus \Def(E) \to H^1(Y\times E,T_E) \to 0,
\end{equation*}
defined above, is short exact.
\end{lem}

\begin{proof}
    We first decompose $H^1(Y\times E,T_{Y\times E})=H^1(Y\times E,T_Y\oplus T_E)$ into its K\"{u}nneth components:
    \begin{align*}
        H^1(Y\times E,T_{Y\times E})&=(H^1(Y,T_Y)\otimes H^0(E,\cO_E)\oplus H^0(Y,T_Y)\otimes H^1(E,\cO_E))\\
        &\oplus(H^1(Y,\cO_Y)\otimes H^0(E,T_E)\oplus H^0(Y,\cO_Y)\otimes H^1(E,T_E)).
    \end{align*}

Consider now the map $H^1(Y\times E,T_{Y\times E})\to H^2(Y\times E,\cO_{Y\times E})$ given by cupping with $(\theta_Y,d\cdot\theta_E)\in H^1(Y,T_Y^\vee)\oplus H^1(E,T_E^\vee)=H^1(Y\times E,T_{Y\times E}^\vee)$ on the right. This map decomposes into the three maps
    \begin{align*}
        H^1(Y,T_Y)\otimes H^0(E,\cO_E)&\to H^2(Y,\cO_Y)\otimes  H^0(E,\cO_E)\\
        (H^0(Y,T_Y)\otimes H^1(E,\cO_E))\oplus (H^1(Y,\cO_Y)\otimes H^0(E,T_E)) &\to H^1(Y,\cO_Y)\otimes H^1(E,\cO_E)\\
        H^0(Y,\cO_Y)\otimes H^1(E,T_E) &\to H^0(Y,\cO_Y)\otimes  H^2(E,\cO_E)=0.
    \end{align*}
The first map, given by cupping with $\theta_Y$ in the first factor, has kernel $H^1(Y,T_Y)^{\alg}\otimes H^0(E,\cO_E)$, and the third has kernel $H^0(Y,\cO_Y)\otimes H^1(E,T_E)$. The second map is an isomorphism on each summand; it acts by cupping with $\theta_Y$ on the first summand (and correcting by a sign) and cupping with $d\cdot\theta_E$ in the second. Its kernel may therefore be identified with either summand; we identify it with $H^1(Y,\cO_Y)\otimes H^0(E,T_E)$.

Therefore, we conclude that
\begin{align*}
    \Def(Y\times E)&\cong (H^1(Y,T_Y)^{\alg}\otimes H^0(E,\cO_E))\\
    &\oplus (H^1(Y,\cO_Y)\otimes H^0(E,T_E))\oplus (H^0(Y,\cO_Y)\otimes H^1(E,T_E)).
\end{align*}

Let $(t_Y,t_E)\in \Def(Y)\oplus\Def(E)$ be a pair of deformations. Then, the map $\Def(Y)\oplus\Def(E)\to \Def(Y\times E)\oplus\Def(E)$ sends $(t_Y,t_E)$ to
\begin{equation*}
    ((t_y\otimes 1,0,d\otimes t_E),t_E)\in \Def(Y\times E)\oplus\Def(E)
\end{equation*}
using the decomposition of $\Def(Y\times E)$ above; this map is clearly injective. On the other hand, the map $\Def(Y\times E)\oplus\Def(E)\to H^1(Y\times E,T_E)$ sends the first summand of $\Def(Y\times E)$ to zero, and identifies the two summands 
\begin{equation*}
(H^1(Y,\cO_Y)\otimes H^0(E,T_E))\oplus (H^0(Y,\cO_Y)\otimes H^1(E,T_E))
\end{equation*}
with the target via the K\"{u}nneth decomposition. It furthermore sends $t_E\in \Def(E)$ to $-d\otimes t_E\in H^0(Y,\cO_Y)\otimes H^1(E,T_E)$. This identifies $\Def(Y\times E)\oplus\Def(E)\to H^1(Y\times E,T_E)$ with the cokernel of the first map.
\end{proof}

We now define the remaining maps in \eqref{defdiagram}, the vertical arrows. The map $\Def(h)\to \Def(Y)\oplus\Def(E)$ is the canonical one remembering the induced deformations of $Y,E$ from a deformation of $h$. The map $\Def(X)\oplus \Def(E)\to \Def(Y\times E) \oplus \Def(E)$ is given by the identity in the second factor and by the map 
\begin{equation*}
    H^1(X,T_X)\to H^1(X,j_{*}j^{*}T_X)\to H^1(X,j_{*}T_{Y\times E}) \to H^1(Y\times E,T_{Y\times E})
\end{equation*}
in the first. The map $H^1(X,j_{*}j^{*}T_X)\to H^1(X,j_{*}T_{Y\times E})$ is induced by (the inverse of) the canonical isomorphism $T_{Y\times E}\cong j^{*}T_X$. The naturality of the cup product shows that the above map restricts to a map $H^1(X,T_X)^{\alg}\to H^1(Y\times E,T_{Y\times E})^{\alg}$. The final vertical map is
\begin{equation*}
    H^1(X,h^{*}T_E)\to H^1(X,j_{*}j^{*}h^{*}T_E)\to H^1(X,j_{*}T_E)\to H^1(Y\times E,T_E).
\end{equation*}

\begin{lem}
    The diagram \eqref{defdiagram} is commutative. 
\end{lem}

\begin{proof}
Consider first the square
\begin{equation*}
\xymatrix{ 
    \Def(Y)\oplus\Def(E) \ar[r] & \Def(Y\times E) \oplus \Def(E)\\
     \Def(h) \ar[u] \ar[r] & \Def(X)\oplus \Def(E) \ar[u]
}.
\end{equation*}
A deformation of $h:X\to E$ maps to the induced deformation $t_E$ of $E$ under both maps to $\Def(E)$. Under the map $\Def(h)\to \Def(X)\to \Def(Y\times E)$, we obtain a deformation of $Y\times E$ that deforms $Y$ and $E$ separately. The deformation $t_Y$ of $Y=\ker(h)^\circ$ is that induced by $h$, but the deformation of $E$ is not $t_E$, but rather that obtained from the \emph{dual} $E^\vee\to X^\vee$. Because the composition $E^\vee\to X^\vee\cong X\to E$ is multiplication by $d$, we therefore obtain $d\cdot t_E\in \Def(E)$. By definition, this agrees with the map $\Def(h)\to\Def(Y)\oplus\Def(E)\to\Def(Y\times E)$.

The square
\begin{equation*}
    \xymatrix{\Def(Y\times E) \oplus \Def(E) \ar[r] & H^1(Y\times E,T_E) \\
     \Def(X)\oplus \Def(E) \ar[r] \ar[u] & H^1(X,h^{*}T_E) \ar[u] }
\end{equation*}
is commutative simply because all maps are induced from the natural ones on the level of sheaves. In particular, note by definition that we have multiplied the natural maps $\Def(E)=H^1(E,T_E)\to H^1(Y\times E,T_E)$ and $\Def(E)=H^1(E,T_E)\to H^1(X,h^{*}T_E)$ by $-d$. 
\end{proof}

\begin{lem}\label{isogeny_deformations}
    Let $B$ be a polarized abelian variety and let $j: A \to B$ be an isogeny. Regard $A$ as a polarized abelian variety via the pullback of the polarization on $B$. Let $\Def(j)$ be the space of deformations of $j$, as a map of polarized abelian varieties. Then, we have a commutative diagram

\begin{equation*}
\xymatrix{
 & \Def(j) \ar[rd] \ar[ld] & \\
\Def(B) \ar[rr] & & \Def(A)
}
\end{equation*}
where the bottom arrow is induced by the pullback $H^1(A,T_A)\to H^1(B,j^{*}T_A)\cong H^1(B,T_B)$ and the diagonal arrows are the forgetful maps. Furthermore, all three arrows are isomorphisms.
\end{lem}

\begin{proof}
We first show that the bottom arrow is an isomorphism. The natural map $H^1(A,T_A)\to H^1(B,T_B)$ preserves algebraic deformations by naturality of the cup product. Moreover, a choice of trivialization $\cO_B^g\cong T_B$ yields a trivialization $\cO_A^g\cong T_A$ upon pullback. Via these compatible trivializations, the map $H^1(A,T_A)\to H^1(B,T_B)$ is induced by the isomorphism $H^1(B,\cO_B)\to H^1(A,\cO)$ (which, in turn, is induced by the isomorphism $H^1(B,\bC)\to H^1(A,\bC)$ on singular cohomology and the Hodge decomposition).

We next describe the inverse map $\Def(B)\to\Def(j)$. The data of $j$ is equivalent to that of  a line bundle $\cL_j$ on $A\times B^\vee$, trivialized along $0\times B^\vee$ and $A\times 0$. This line bundle is the pullback of the Poincar\'{e} bundle $\cL_\cP$ on $B\times B^\vee$ via the map 
\begin{equation*}
(j,\id): A\times B^\vee \to B\times B^\vee.
\end{equation*}
Given any class $t_B\in H^1(B,T_B)$ corresponding to a deformation $\mathcal B \to S$ of $B$, taking relative $\Pic^0$ gives a deformation $\mathcal B^\vee \to S$ of $B^\vee$, which in turn corresponds to a deformation $t_{B^\vee}\in H^1(B^\vee,T_{B^\vee})$.

Because the Poincar\'{e} bundle $\cL_\cP$ extends to a relative Poincar\'{e} bundle over the product deformation $\cB\times_S \cB^\vee$, we obtain that the cup product map
\begin{equation*}
H^1(B\times B^\vee,T_{B\times B^\vee})\otimes H^1(B\times B^\vee,\Omega_{B\times B^\vee}) \to H^2(B\times B^\vee,\cO_{B\times B^\vee})
\end{equation*}
sends $(t_B,t_{B^\vee})\otimes c_1(\cL_\cP)$ to zero. Observe that $(t_B,t_{B^\vee})$ is supported on the K\"{u}nneth components
\begin{equation*}
H^1(B,T_B)\otimes H^0(B^\vee, \cO_{B^\vee}) \oplus H^0(B, \cO_B)\otimes H^1(B^\vee, T_{B^\vee})
\end{equation*}
of $H^1(B\times B^\vee,T_{B\times B^\vee})$, and that $c_1(\cL_\cP)$ is supported on the K\"{u}nneth components
\begin{equation*}
    H^0(B,\Omega_B)\otimes H^1(B^\vee, \cO_{B^\vee}) \oplus H^1(B, \cO_B)\otimes H^0(B^\vee, \Omega_{B^\vee})
\end{equation*}
of $H^1(B\times B^\vee,\Omega_{B\times B^\vee})$, because $\cL_\cP$ is trivialized along both zero sections. Letting $(c_B^0,c_B^1)$ be the projections of $c_1(\cL_\cP)$ to the two summands above, we have $t_B\cdot c_B^0 + t_{B^\vee}\cdot c_{B}^1=0$. 

Now, let $t_A$ be the image of $t_B$ under the map $\Def(B)\to \Def(A)$; we consider the cup product $(t_A,t_{B^\vee})\otimes c_1(\cL_j)$. We have that $(t_A,t_{B^\vee})$ is supported on the K\"{u}nneth components
\begin{equation*}
H^1(A,T_A)\otimes H^0(B^\vee, \cO_{B^\vee}) \oplus H^0(A, \cO_A)\otimes H^1(B^\vee, T_{B^\vee})
\end{equation*}
of $H^1(A\times B^\vee,T_{A\times B^\vee})$, given by pullback from $H^1(B\times B^\vee,T_{B\times B^\vee})$, and that the projections $(c_A^0,c_A^1)$ of $c_1(\cL_j)$ to the K\"{u}nneth components
\begin{equation*}
    H^0(A,\Omega_A)\otimes H^1(B^\vee, \cO_{B^\vee}) \oplus H^1(A, \cO_A)\otimes H^0(B^\vee, \Omega_{B^\vee})
\end{equation*}
of $H^1(A\times B^\vee,\Omega_{A\times B^\vee})$ are also given by pullback by $j$. By naturality of cup product, the product $t_A\cdot c^0_A + t_{B^\vee}\cdot c^1_{A}=0$ is equal to zero. Therefore, the line bundle $\cL_j$ extends to the deformation $\cA\times\cB^\vee$ corresponding to $(t_A,t_{B^\vee})$, giving a deformation of $j$ and therefore the desired natural map $\Def(j)\to\Def(B)$.

The composition $\Def(B)\to\Def(j)\to\Def(B)$ is clearly equal to the identity by construction. Moreover, the kernel of the map $\Def(j)\to\Def(B)$ is the space of deformations of $j$ with the target $B$ fixed, which is equal to $H^1(A,T_{A/B})=0$ by \cite[Lemma 3.4.7]{sernesi}. Therefore, the map $\Def(j)\to \Def(B)$ is an isomorphism, with inverse given by the map $\Def(B)\to\Def(j)$ constructed above. Moreover, the composition $\Def(B)\to\Def(j)\to\Def(A)$ is equal to the pullback map on deformations by construction, so we obtain the needed commutativity. 
\end{proof}

\begin{proof}[Proof of Proposition \ref{Def(h)}]
It remains to show that the vertical maps of \eqref{defdiagram} are isomorphisms. The middle vertical map is an isomorphism by Lemma \ref{isogeny_deformations}. The map $H^1(X,h^{*}T_E)\to H^1(Y\times E,h^{*}T_E)$ is similarly seen to be an isomorphism, after choosing a trivialization $\cO_E\cong T_E$.
%First, after choosing an isomorphism $\cO_X^g\cong T_X$, which pulls back under $g$ to an isomorphism $\cO_{Y\times E}^g\cong T_{Y\times E}$, the map $H^1(X,T_X)\to H^1(Y\times E,T_{Y\times E})$ is given simply by $H^1(\cO_X)^g\to H^1(\cO_{Y\times E})^g$, where the map $H^1(\cO_X)\to H^1(\cO_{Y\times E})$ is the dual map on tangent spaces at the identity. Restricting to the algebraic parts on each side and adding back the factor of $\Def(E)$ shows that the middle map is an isomorphism. Similarly, choosing a trivialization $\cO_E\cong T_E$ identifies the map $H^1(X,h^{*}T_E)\to H^1(Y\times E,T_E)$ with the isomorphism on cotangent spaces $(\Omega_X)_0\to (\Omega_{Y\times E})_0$.

Finally, an inverse to the map $\Def(h)\to\Def(Y)\oplus\Def(E)$ is given as follows. Deformations of $Y,E$ give rise to a deformation of $Y\times E$ via the map $\Def(Y)\times\Def(E)\to\Def(Y\times E)$ from before. Then, applying Lemma \ref{isogeny_deformations} to the induced isogeny $j:Y\times E\to X$ yields a deformation of $j$, which in turn yields a deformation of $h$ upon dualizing and projecting to $E$. The composition $\Def(Y)\oplus\Def(E)\to\Def(h)\to\Def(Y)\oplus\Def(E)$ is the identity by construction. Furthermore, an element of the kernel of $\Def(h)\to\Def(Y)\oplus\Def(E)$ must restrict to the trivial deformation on $X$ in addition to that of $Y$ and $E$, by Lemma \ref{isogeny_deformations}. However, the space of deformations of $h$ fixing $X,E$ is given by $H^0(X,f^{*}T_E)$, which is generated by the deformation obtained by post-composition on $E$, so there are no non-trivial deformations of $h$ as a map of abelian varieties. This completes the proof.
\end{proof}

\begin{proof}[Proof of Proposition \ref{NL_normal_bundles}]
    Proposition \ref{Def(h)} globalizes to a short exact sequence
    \begin{equation*}
        0\to \cT_{\wNL_{g,d}} \to (\epsilon,\mu)^{*}\cT_{\cA_g\times\cA_1}\to R^1(\pi_{\NL,g})_{*}h_{\NL}^{*}T_{\cX_1/\cA_1}\to 0.
    \end{equation*}
\end{proof}

\subsection{Comparison to stable maps}

\begin{defn}
Let $\cM^{\ct}_{g,1}$ be the stack of pointed curves of compact type. Let $\Tor_1:\cM_{g,1}^{\ct}\to\cA_g$ be the Torelli map. 
\end{defn}

We work throughout this section with the open substack $\cM^{\ct,q}_{g,1}(\cE,d)\subset \barM^{q}_{g,1}(\cE,d)$ of stable maps (up to translation) whose source curve is of compact type. The virtual class $[\cM^{\ct,q}_{g,1}(\cE,d)]^{\trunc}\in A_{2g-1}(\cM^{\ct,q}_{g,1}(\cE,d))$ is defined by restriction.

\begin{lem}\label{degrees_agree}
Fix $d>0$ and $[f:C\to E]\in \cM^{\ct,q}_{g,1}(\cE,d)$. Then, the induced map $f^{*}:E\to J(C)$ has degree $d$.
\end{lem}

\begin{proof}
The composition $(f^{*})^\vee\circ(f^{*}):E\to J(C)\to E$ is canonically identified with $f_{*}\circ f^{*}$, which is multiplication by $d$.
\end{proof}

In particular, we obtain a map $\wt{\Tor}_1:\cM^{\ct,q}_{g,1}(\cE,d)\to \wNL_{g,d}$.

\begin{prop}\label{pullback_to_mg1_prop}
We have a commutative diagram
\begin{equation*}
\xymatrix{
\cM^{\ct,q}_{g,1}(\cE,d) \ar[r]^(0.6){\wt{\Tor}_1} \ar[d]_{(\epsilon,\mu)} & \wNL_{g,d} \ar[d]^{(\epsilon,\mu)} \\
\cM_{g,1}^{\ct} \times \cM_{1,1} \ar[r]^(0.54){(\Tor_1,\id)} \ar[d] & \cA_g \times \cA_1 \ar[d] \\
\cM_{g,1}^{\ct}  \ar[r]^(0.54){\Tor_1} & \cA_g
}
\end{equation*}
where both squares are Cartesian.
\end{prop}

\begin{proof} 
The bottom square is obtained simply by taking products with $\cM_{1,1}=\cA_1$. We consider only the top square.

The commutativity is clear. We construct the inverse map
\begin{equation*}
    \gamma:\wNL_{g,d}\times_{\cA_g\times\cA_1}(\cM_{g,1}^{\ct} \times \cM_{1,1})\to \cM^{\ct,q}_{g,1}(\cE,d).
\end{equation*}
Consider a test family $S\to \wNL_{g,d}\times_{\cA_g\times\cA_1}(\cM_{g,1}^{\ct} \times \cM_{1,1})$. The map $S\to\cM^{\ct}_{g,1}$ corresponds to a family of curves $\pi:\cC\to S$ with section $p:S\to \cC$, which in turn gives rise to an Abel-Jacobi morphism $AJ_S:\cC\to J(\cC)$ over $S$. We also have a universal $d$-elliptic map map $h:J(\cC)\to\cE_S$ to the pullback of the universal family $\pi:\cE\to\cM_{1,1}=\cA_1$ over $S$. We now define the family of stable maps $\gamma(S)$ by the composition $h\circ AJ_S:\cC\to \cE_S$ over $S$, which is easily verified to be the desired inverse.
\end{proof}

In particular, we have proven Theorem \ref{thm_pullback_to_mg1}.

\begin{prop}\label{gysin_equals_trunc}
The Gysin pullback $(\Tor_1,\id)^{!}([\wNL_{g,d}])\in A_{2g-1}(\cM^{\ct,q}_{g,1}(\cE,d))$ is equal to the virtual class $[\cM^{\ct,q}_{g,1}(\cE,d)]^{\trunc}$.
\end{prop}

\begin{proof}
Recall that $[\cM^{\ct,q}_{g,1}(\cE,d)]^{\trunc}$ is by definition the virtual class associated to the map
\begin{equation*}
    (R^1(\pi^q_g)_{*}(f^q)^{*}T_{\cE/\cM_{1,1}})^{\vee}[1]\to \bL_{\cM^{\ct,q}_{g,1}(\cE,d)/(\cM_{g,1}^{\ct}\times\cM_{1,1})}.
\end{equation*}
On the other hand, the Gysin pullback $(\Tor_1,\id)^{!}([\wNL_{g,d}])\in A_{2g-1}(\cM^{\ct,q}_{g,1}(\cE,d))$ is the virtual class associated to the map
\begin{equation*}
    \wt{\Tor}_1^{*}C_{\wNL_{g,d}/(\cA_g\times\cA_{1})}[1] \to \bL_{\cM^{\ct,q}_{g,1}(\cE,d)/(\cM_{g,1}^{\ct}\times\cM_{1,1})},
\end{equation*}
where $C_{\wNL_{g,d}/(\cA_g\times\cA_{1})}$ denotes the conormal bundle. It therefore suffices to identify these two maps.

% Recall from \S\ref{kontsevich_family} that the universal stable map $f^q:\cC_{g,1}^{\ct,q}(\cE,d)\to \cE$ induces a map 
% \begin{equation*}
%     (R(\pi^q_g)_{*}(f^q)^{*}T_{\cE/\cM_{1,1}})^{\vee}\to \bL_{\cM^{\ct,q}_{g,1}(\cE,d)/(\cM_{g,1}^{\ct}\times\cM_{1,1})}
% \end{equation*}
By Corollary \ref{NL_normal_bundles}, we have a canonical isomorphism of sheaves in degree $-1$
\begin{equation*}
e_{\trunc}:(R^1(\pi_{\NL,g})_{*}h_{\NL}^{*}T_{\cX_1/\cA_1})^{\vee}[1]\to C_{\wNL_{g,d}/(\cA_g\times\cA_{1})}[1].
\end{equation*}
Here, we identify $\cA_1$ with $\cM_{1,1}$ and $\cX_1$ with $\cE$. On the other hand, the universal Abel-Jacobi map $AJ:\cC\to J(\cC)$ induces a map 
%The same construction as in \S\ref{kontsevich_family} and \S\ref{1-pointed_sec} gives a map
% \begin{equation*}
% (R^1(\pi_{\NL,g})_{*}h_{\NL}^{*}T_{\cX_1/\cA_1})^{\vee}[1]\to \bL_{\wNL_{g,d}/(\cA_g\times\cA_{1})}
% \end{equation*}
% which is an isomorphism in degree $-1$ by Corollary \ref{NL_normal_bundles}
\begin{equation*}
    \wt{\Tor}_1^{*}(R^1(\pi_{\NL,g})_{*}h_{\NL}^{*}T_{\cX_1/\cA_1})^{\vee}[1]\to (R^1(\pi^q_g)_{*}(f^q)^{*}T_{\cE/\cM_{1,1}})^{\vee}[1]
\end{equation*}
which is an isomorphism in degree $-1$. We obtain a commutative square
\begin{equation*}
    \xymatrix{
    \wt{\Tor}_1^{*}(R^1(\pi_{\NL,g})_{*}h_{\NL}^{*}T_{\cX_1/\cA_1})^{\vee}[1] \ar[r]^(0.53){\sim}
     \ar[d]_{\sim} & (R^1(\pi^q_g)_{*}(f^q)^{*}T_{\cE/\cM_{1,1}})^{\vee}[1] \ar[d] \\
    \wt{\Tor}_1^{*}C_{\wNL_{g,d}/(\cA_g\times\cA_{1})}[1] \ar[r] & \bL_{\cM^{\ct,q}_{g,1}(\cE,d)/(\cM_{g,1}^{\ct}\times\cM_{1,1})}
    }
\end{equation*}
which gives the needed identification.
% The Gysin pullback $(\Tor_1,\id)^{!}([\wNL_{g,d}])\in A_{2g-1}(\cM^{\ct,q}_{g,1}(\cE,d))$ is defined as the virtual class associated to the bottom horizontal map, again after restricting to degree $-1$. On the other hand, recall that the virtual class $[\cM^{\ct,q}_{g,1}(\cE,d)]^{\trunc}$ is defined from the obstruction theory defined by the right vertical map $e_{\trunc}$. The commutative square identifies these two maps in degree $-1$. Recall moreover from the proof of Proposition \ref{trunc_perfobs} that the map $e_{\trunc}$ is zero in degree zero, so the virtual class $[\cM^{\ct,q}_{g,1}(\cE,d)]^{\trunc}$ is in fact equal to that obtained after restricting to degree $-1$. Thus, the two classes in question coincide.
\end{proof}

\begin{cor}\label{cor_tor1_pullback}
    The Gysin pullback $\Tor_1^{!}([\wNL_{g,d}])\in A_{2g-1}(\cM^{\ct,q}_{g,1}(\cE,d))$ is equal to the virtual class $[\cM^{\ct,q}_{g,1}(\cE,d)]^{\trunc}$.
\end{cor}

\begin{proof}
    Immediate from Proposition \ref{gysin_equals_trunc} and \cite[Theorem 6.2(c)]{fulton}.
\end{proof}

In particular, Theorem \ref{main_thm} follows from Proposition \ref{trunc_gw_comparison}.

\begin{rem}\label{rem_unpointed}
    One can also consider the Gysin pullback of $[\wNL_{g,d}]$ by the \emph{unpointed} Torelli map $\Tor:\cM^{\ct}_g\to\cA_g$. The diagram
    \begin{equation*}
\xymatrix{
\cM^{\ct}_{g}(\cE,d)/\cE \ar[r]^(0.6){\wt{\Tor}} \ar[d]_{\epsilon} & \wNL_{g,d} \ar[d]^{\epsilon} \\
\cM_{g}^{\ct}  \ar[r]^(0.54){\Tor} & \cA_g
}
\end{equation*}
is Cartesian, where $\cM^{\ct}_{g}(\cE,d)/\cE$ is the stack of unpointed stable maps to a (varying) smooth elliptic curve $E$ \emph{up to translation on the target}. The Gysin pullback $\Tor^{!}([\wNL_{g,d}])$ may again be identified with the natural reduced virtual class $[\cM^{\ct}_{g}(\cE,d)/\cE]^{\red}$ globalizing that considered in \cite{bopy}. Moreover, under the map
\begin{equation*}
    \phi:\cM^{\ct,q}_{g,1}(\cE,d)\to \cM^{\ct}_{g}(\cE,d)/\cE,
\end{equation*}
we have the compatibility 
\begin{equation*}
    \phi_{*}([\cM^{\ct,q}_{g,1}(\cE,d)]^{\trunc}\cap \psi)=(2g-2)\cdot [\cM^{\ct}_{g}(\cE,d)/\cE]^{\red},
\end{equation*}
essentially by definition of $[\cM^{\ct}_{g}(\cE,d)/\cE]^{\red}$. We omit the proofs, but one may follow the same arguments as above, with some additional descent considerations to deal with the quotient $\cM^{\ct}_{g}(\cE,d)/\cE$.
\end{rem}

\end{document}